\documentclass[11pt, fleqn]{article}

\usepackage{eucal} 
\usepackage{amsmath}
\usepackage{amsthm}
\usepackage{amssymb}
\usepackage{enumerate}
\usepackage{a4wide} 
\usepackage{bbm} 
\usepackage{cite}

\newcommand{\core}[1]{{#1}^{\tiny{\textcircled{\#}}}} 
\newcommand{\dcore}[1]{{#1}_{\tiny{\textcircled{\#}}}} 
\newcommand{\ce}{\mathbbmss{C}} 
\newcommand{\ene}{\mathbbmss{N}} 
\newcommand{\ere}{\mathbbmss{R}} 
\newcommand{\uno}{\mathbbmss{1}} 
\newcommand{\nn}{\mathcal{N}} 
\newcommand{\rr}{\mathcal{R}} 
\renewcommand{\aa}{\mathcal{A}} 
\newcommand{\hh}{\mathcal{H}} 
\newcommand{\xx}{\mathcal{X}} 
\newcommand{\mm}{\mathcal{M}} 
\renewcommand{\ll}{\mathcal{L}} 
\newcommand{\ttt}{\mathcal{T}}
\newcommand{\sss}{\mathcal{S}}
\newcommand{\gap}[2]{\widehat{\delta} \left(#1,#2 \right)}
\DeclareMathOperator{\rk}{rk} 
\DeclareMathOperator{\ds}{dist} 

\newcommand{\suc}[1]{(#1)_{n \in \ene}}

\newtheorem{df}{Definition}[section]
\newtheorem{theorem}[df]{Theorem}
\newtheorem{lemma}[df]{Lemma}
\newtheorem{remark}[df]{Remark}
\newtheorem{example}[df]{Example}
\newtheorem{proposition}[df]{Proposition}
\newtheorem{corollary}[df]{Corollary}

\title{On the continuity and differentiability\\ of the (dual) core inverse in $C^*$-algebras}
\date{}
\author{Julio Ben\'{\i}tez, Enrico Boasso, Sanzhang Xu}

\begin{document}

\maketitle

\begin{abstract}\noindent The continuity of the core inverse and the dual core inverse is studied in the setting of $C^*$-algebras. Later, 
this study is specialized to the case of bounded Hilbert space operators and 
to complex matrices. In addition, the differentiability of these generalized inverses is studied in the context of $C^*$-algebras.
\end{abstract}

\section{Introduction}

The core inverse and the dual core inverse of a matrix were introduced in \cite{BT}. These generalized inverses have been studied by several authors,
in particular they have been extended to rings with involution (\hspace{-1pt}\cite{rdd}) and to Hilbert space operators (\hspace{-1pt}\cite{rdd2}). It is worth noticing that the inverses under consideration 
are closely related to the group inverse and the Moore-Penrose inverse; 
to learn more results concerning these notions, see for example \cite{BT, rdd, rdd2, 14}.

So far the properties of the (dual) core inverse that have been researched are mainly of algebraic nature and the setting has been essentially the one of rings with involution.
The objective of the present article is to study the continuity and the differentiability of these inverses in the context of $C^*$-algebras.

In fact, in section 3, after having recalled several preliminary results in section 2, the continuity of the core inverse and of  the dual core inverse will be studied. Two main characterizations will be 
presented. The first one relates the continuity of the aforementioned notions to the continuity  of the group inverse and of the Moore-Penrose inverse. The second characterization 
uses the notion of the gap between subspaces; a similar approach has been used to study the continuity of the Drazin inverse and of the Moore-Penrose inverse, see for example \cite{V, V2, kr1}
and \cite[Chapter 4]{dr}. In section 4 results regarding the continuity of the (dual) core inverse of Hilbert space operators and matrices will be presented. Finally, in section 5 the differentiability of the generalized inverses under consideration will be researched. Furthermore, some results concerning the continuity and the differentiability of the group inverse and the Moore-Penrose inverse
will be also proved.

It is  noteworthy to mention that the core inverse and the dual core inverse are two particular cases of the $(b, c)$-inverse (\hspace{-1pt}\cite{D}), see \cite[theorem 4.4]{rdd}.  
Therefore the representations and other results presented in \cite[Section 7]{b2} can be applied to these generalized inverses.

\section{Preliminary Definitions}

\noindent Since properties of $C^*$-algebra elements will be studied in what follows, although the main notions considered in this article 
can be given in the context of rings with involution,
all the definition will be presented in the frame of $C^*$-algebras. 

 From now on $\aa$ will denote a unital $C^*$-algebra with unity $\uno$. In addition, $\aa^{-1}$ will stand for the set of all invertible elements in $\aa$.
Given $a\in\aa$, the {\em image ideals} and the {\em null ideals} defined by $a\in\aa$ 
are the following sets:
\begin{align*}
& &&a \aa = \{ ax: x \in \aa \},& &\aa a = \{ xa: x \in \aa \},&\\ 
& &&a^\circ = \{ x \in \aa: ax=0 \},& &{}^\circ a = \{ x \in \aa: xa=0 \}.&&\\
\end{align*}
\noindent Recall that $a\in\aa$  is said to be \it regular, \rm if there exists $b\in \aa$ such that $a=aba$. In addition,
$b\in\aa$ is said to be \it an outer inverse of $a\in\aa$, \rm if $b=bab$.

The notion of invertible element has been generalized or extended in several ways.   
One of the most important notion of generalized inverse is the Moore-Penrose inverse. An element $a\in\aa$ is said to be \it Moore-Penrose
invertible, \rm if there is $x\in\aa$ such that the following equations hold:

\begin{align*} 
& &&axa=a,& &xax=x,& &(ax)^*=ax,& &(xa)^* = xa.&\\
\end{align*}

\noindent It is well known that if such an $x$ exists, then it is unique, and in this case $x$, the Moore-Penrose inverse of $a$,
will be denoted by $a^\dag$. Moreover, the subset of $\aa$ composed of all the
Moore-Penrose invertible elements of $\aa$ will be denoted by $\aa^\dagger$. It is worth noticing that according to \cite[Theorem 6]{hm1}, a necessary and sufficient condition for
$a\in \aa^\dagger$ is that $a\in\aa$ is regular, which in turn is equivalent to $a\aa$ is closed (\hspace{-1pt}\cite[Theorem 8]{hm1}). Moreover, if $a\in\aa^\dagger$, then it is not difficult to prove that
$a^\dag \aa=a^* \aa$ and $ \aa a^\dag=\aa a^*$. To learn more properties of the Moore-Penrose inverse in the frame of $C^*$-algebras, see \cite{hm1, hm, P, mb, rs, dr}.\par

Another generalized inverse which will be central for the purpose of this article is the group inverse.
An element $a \in \aa$
is said to be {\em group invertible}, if there is $x \in \aa$ such that
\begin{align*}
& & &axa=a,&  &xax=x,& &ax=xa.&
\end{align*}
It can be easily proved that if such $x$ exists, then it is unique. The group inverse is
customarily denoted by $a^\#$. The subset of $\aa$ composed by all the group invertible elements in $\aa$
will be denoted by $\aa^\#$.

Next follows one of the main notions of this article (see \cite[Definition~2.3]{rdd}, see also \cite{BT} for the original definition in the context of matrices). 

\begin{df} \label{df1}Given a unital $C^*$-algebra $\aa$, an element $a \in \aa$ will be said to be core invertible,
if there exists  $x \in \aa$ such that the following equalities hold:
$$
axa=a, \qquad x \aa = a \aa, \qquad \aa x = \aa a^*.
$$
\end{df}

According to \cite[Theorem 2.14]{rdd}, if such an element $x$ 
exists, then it is unique. This element will be said
to be the {\em core inverse} of $a \in \aa$ and it will be denoted by $\core{a}$. In addition,
the set of all core invertible elements of $\aa$ will be denoted
by $\core{\aa}$.

Recall that according to \cite[Theorem 2.14]{rdd}, when $\core{a}$ exists ($a\in\aa$), it is an outer inverse of $a$, i.e.,
$\core{a} a \core{a} = \core{a}$. Moreover, in
\cite[Theorem 2.14]{rdd}, the authors characterized 
the core invertibility in terms of equalities. This characterization 
was improved in  \cite[Theorem~3.1]{14}. Specifically, $a \in \aa$ is core invertible if and only if there exists $x \in \aa$ such that
\begin{align*} \label{prop_core}
& &&a x^2 = x,& &x a^2 = a,& &(a x)^* = a x.&\\
\end{align*}
Furthermore, if such $x$ exists, then $x = \core{a}$.

Another generalized inverse, which is related with the core inverse, was defined in \cite{rdd}. 

\begin{df}\label{df2} Given $\aa$ a unital $C^*$-algebra, an element $a \in \aa$ is said to be  dual 
core invertible, if there is  $x \in \aa$ such that $axa=a$, $x\aa=a^* \aa$, and $\aa x = \aa a$.
\end{df}

As for the core inverse, it can be proved that this $x$ is unique, when it exists; thus 
it will be denoted by $\dcore{a}$ and $\dcore{\aa}$ will stand for the set of all dual core invertible elements of $\aa$. 
Note that $a \in \aa$ is core invertible 
if and only if $a^*$ is dual core invertible and in this case, 
$\dcore{(a^*)}=(\core{a})^*$. In addition, according to \cite[Theorem 2.15]{rdd}, when $a\in\dcore{\aa}$, $\dcore{a}$ is an outer inverse of $a$,
i.e., $\dcore{a}=\dcore{a} a\dcore{a}$. 

Observe that according to Definition \ref{df1} (respectively Definition \ref{df2}), if $a\in\aa$ is core invertible
(respectively dual core invertible), then it is regular, and hence $a$ is Moore-Penrose invertible (\hspace{-1pt}\cite[Theorem 6]{hm1}). Moreover, 
if $a\in \core{\aa}\cup\dcore{\aa}$, then $a$ is group invertible (\hspace{-1pt}\cite[Remark 2.16]{rdd}). Furthermore, 
according to \cite[Theorem 2.19]{rdd}, if $a\in\core{\aa}$, then the following equalities hold:
\begin{align*}
&a^\#=(\core{a})^2a,& &a^\#=\core{a}a\dcore{a},& &a^\dag=\dcore{a}a\core{a},&\core{a}=a^\# a a^\dag,& &\dcore{a}=a^\dag a a^\#.&\\ 
\end{align*}

\noindent To learn more on the properties of the core and dual core inverse,
see \cite{BT, rdd, 14}.

On the other hand, $\xx$ will stand for  a Banach space and $\ll(\xx)$ for the algebra of all operators defined on and with values in $\xx$. 
When $A \in \ll(\xx)$, the range 
and the null space of $A$ will be denoted by $\rr(A)$ and $\nn(A)$, respectively.
When $\dim \xx < \infty$ and $A \in \ll(\xx)$, the dimension of $\rr(A)$ will be denoted
with $\rk(A)$. Evidently, if $A\in\ce_n$, the set of complex $n \times n$ matrices, by considering
that $A \in \ll(\ce^n)$, the rank of the complex matrix $A$ coincides with the
previously defined $\rk(A)$; consequently, the same notation will be used for both notions.

One of most studied generalized inverses is the outer inverse with prescribed range
and null space. This generalized inverse will be introduced in the Banach frame. 
Let $\xx$ be a Banach space and consider $A \in \ll(\xx)$ and $\ttt$, $\sss$  two closed subspaces
in $\xx$. If there exists an operator $B \in \ll(\xx)$ such that $BAB=B$, 
$\nn(B)=\sss$, and $\rr(B)=\ttt$, then 
such $B$ is unique (\hspace{-1pt}\cite[Theorem 1.1.10]{dr}). 
In this case, $B$ will be said
to be the {\em $A_{\ttt,\sss}^{(2)}$ outer inverse of} $A$.

To prove several results of this article, the definition of the gap between two subspaces
need to be recalled. Let $\xx$ be a Banach space and consider $\mm$ and $\nn$
two closed subspaces in $\xx$. If $\mm=0$, then set $\delta(\mm,\nn)=0$, otherwise set
\begin{align*}
& & &\delta(\mm,\nn) = \sup \{ \ds (x, \nn) : x \in \mm, \| x \| = 1 \},&\\
\end{align*}
where $\ds(x,\nn)=\inf \{ \| x - y \| : y \in \nn \}$. The {\em gap between the closed
subspaces $\mm$ and $\nn$} is
\begin{align*}
& &&\gap{\mm}{\nn} = \max \{ \delta(\mm,\nn), \delta(\nn,\mm) \}.&\\
\end{align*}
See \cite{dr,dx,kato} for a deeper insight of this concept.

Another notion needed to study the continuity of the (dual) core inverse is the following.
Let $p$ and $q$ be self-adjoint idempotents in a 
$C^*$-algebra $\aa$. The {\em maximal angle} between $p$ and $q$ is the number 
$\psi(p, q) \in [0, \pi/2]$ such that $\| p - q \| = \sin \psi(p, q)$; see \cite[Definition 2.3]{cbl2}.
In what follows, given $x\in\aa^\dag$, $\psi_x$ will stand for the maximal angle between $xx^\dag$ and
$x^\dag x$, i.e., $\psi_x=\psi (xx^\dag, x^\dag x)$.

\section{Continuity of the (dual) core inverse }

In first place a preliminary result need to be presented.

\begin{theorem}\label{theo310} Let $\aa$ be a unital $C^*$-algebra and consider $a\in\aa$. The following statements are equivalent
\begin{enumerate}[{\rm (i)}]
\item $a$ is core invertible.
\item $a$ is dual core invertible.
\item $a^*$ is core invertible.
\item $a^*$ is dual core invertible.
\item $a$ is group invertible and Moore-Penrose invertible.
\end{enumerate}
\noindent In particular, $\core{\aa}=\dcore{\aa}=\aa^\#\cap\aa^\dag=\aa^\#$.
\end{theorem}
\begin{proof} The equivalence between statements (i) and (iv) and between statements (ii) and (iii) can be derived from Definition \ref{df1}
and Definition \ref{df2}. Note that to conclude the proof, it is enough to prove the last statement of the Theorem. In fact, this statement implies that statement (i) and (ii) are equivalent.

According to \cite[Remark 2.16]{rdd}, $\core{\aa}\cup\dcore{\aa}\subseteq \aa^\#$. Moreover, according to \cite[Theorem 6]{hm1},
 $\core{\aa}\cup\dcore{\aa}\subseteq \aa^\dag$. Therefore, according to \cite[Remark 2.16]{rdd}, 
\begin{align*}
&\core{\aa}\subseteq \aa^\#\cap \aa^\dag=\core{\aa}\cap \dcore{\aa}\subseteq \core{\aa},&
&\dcore{\aa}\subseteq \aa^\#\cap \aa^\dag=\core{\aa}\cap \dcore{\aa}\subseteq \dcore{\aa}&\\
\end{align*}
\noindent Finally, according to \cite[Theorem 6]{hm1}, $ \aa^\#\cap \aa^\dag= \aa^\#$.
\end{proof}

Note that under the same conditions in Theorem \ref{theo310}, as for the group inverse and the Moore-Penrose inverse,
$(\core{\aa})^*=\core{\aa}$ and $(\dcore{\aa})^*=\dcore{\aa}$, where if $X\subseteq \aa$ is a set, $X^*$ stands for the following set:
$X^*=\{x^*\colon x\in X\}$. However, in contrast to the case of the group inverse and the Moore-Penrose inverse (when $a^\#$ (respectively $a^\dag$) exits,
$(a^*)^\#=(a^\#)^*$ (respectively $(a^*)^\dag=(a^\dag)^*$), $a\in\aa$),
 recall that to obtain the core inverse (respectively the dual core inverse) of $a^*$ it is necessary to 
consider the dual core (respectively the core) inverse of $a$: 
\begin{align*}
& &&\core{(a^*)}=(\dcore{a})^*,& &\dcore{(a^*)}=(\core{a})^*.&\\
\end{align*}

To prove the first characterization of this section some preparation is needed.

\begin{lemma}\label{lema2}
Let $\aa$ be a unital $C^*$-algebra and consider $a\in\aa$.
\begin{enumerate}[{\rm (i)}]
\item If $a \in \core{\aa}$, then $aa^\dag \core{a} = \core{a}$.
\item If $a \in \core{\aa}$, then $(aa^\dag + a^\dag a - \uno) \core{a} = a^\dag$.
\item Suppose that $a\in\aa$ is regular. The element $aa^\dag + a^\dag a -\uno$ is invertible if and only if $a$ is core invertible. Moreover,
	in this case, $(aa^\dag + a^\dag a -\uno)^{-1} = \core{a}a + (\core{a}a)^* - \uno$.
\item If $a\in\dcore{\aa}$, then $a^\#=a(\dcore{a})^2$.
\end{enumerate}
\end{lemma}
\begin{proof} Recall that according to \cite[Theorem 6]{hm1}, if $a\in \core{\aa}$, then  $a^\dag$ exists.\par

The proof of statement (i) can be derived from the fact that $\core{a} \in a \aa$.

To prove statement (ii), recall that according to \cite[Theorem 2.19 (v)]{rdd}, $\core{a} = a^\# a a^\dag$. Therefore, 
$$(aa^\dag + a^\dag a - \uno) \core{a} = 
aa^\dag \core{a} + a^\dag a \core{a}-\core{a}= a^\dag a \core{a} = 
a^\dag a a^\# a a^\dag = a^\dag.$$

Now statement (iii) will be proved. Note that according to \cite[Theorem 6]{hm1}, $a^\dag$ exists.
Recall that according to \cite[Theorem 2.3]{bc}, $aa^\dag + a^\dag a -\uno\in\aa^{-1}$
is equivalent to $a\in\aa^\#$. Thus, according to Theorem \ref{theo310}, necessary and sufficient for
$aa^\dag + a^\dag a -\uno\in\aa^{-1}$ is that $a\in\core{\aa}$. Next the formula of the inverse of $aa^\dag + a^\dag a -\uno$
will be proved. Recall that according to \cite[Theorem~3.1]{14},  $ \core{a} a aa^\dag= aa^\dag$. 

\begin{equation*}
\begin{split}
[ \core{a} a & + (\core{a}a)^* - \uno ] [ aa^\dag + a^\dag a - \uno ] \\ 
& = 
\left[ \core{a} a + (\core{a}a)^* - \uno \right] aa^\dag + 
\left[ \core{a} a + (\core{a}a)^* - \uno \right] a^\dag a -
\left[ \core{a} a + (\core{a}a)^* - \uno \right] &  \\
& = aa^\dag + (\core{a}a)^*(aa^\dag)^*-aa^\dag + \core{a}a + (\core{a}a)^*(a^\dag a)^*-a^\dag a
- \core{a}a - (\core{a} a)^* + \uno \\
& = (aa^\dag \core{a}a)^* + (a^\dag a \core{a}a)^*-a^\dag a-(\core{a}a)^* +\uno 
&  \\
& = (\core{a}a)^* + (a^\dag a)^*-a^\dag a-(\core{a}a)^* +\uno \\
& = \uno.
\end{split}
\end{equation*}
Since $aa^\dag + a^\dag a - \uno$ is invertible,
$(aa^\dag + a^\dag a - \uno)^{-1} = \core{a} a + (\core{a}a)^* - \uno$.

To prove statement (iv), recall that according Theorem \ref{theo310}, $a^*\in\core{\aa}$. In addition, according to the paragraph between   Theorem \ref{theo310}
and the present Lemma, $\dcore{a}=(\core{(a^*)})^*$. However, according to \cite[Theorem 2.19]{rdd}, $(a^*)^\#=(\core{(a^*)})^2a^*$. Thus,
$$
a^\#=a((\core{(a^*)})^*)^2=a(\dcore{a})^2.
$$ 
\end{proof}

Note that given a ring with involution $\rr$, Lemma \ref{lema2} holds in such a context provided that $a\in\rr$ is Moore-Penrose invertible,

In the next theorem the continuity of the (dual) core inverse will be characterized. It is worth noticing that $a\in\aa$ will be not assumed to be  core invertible, dual core invertible, 
group invertible or Moore-Penrose invertible. 
Note also that the following well known result will be used
in the proof of the theorem: given $\aa$ a unital Banach algebra, $b\in\aa$ and $\suc{b_n}\subset\aa^{-1}$ a sequence such that  $\suc{b_n}$ converges to $b$, if
$\suc{b_n^{-1}}$ is a bounded sequence, then $b$ is invertible and the sequence $\suc{b_n^{-1}}$ converges to $b^{-1}$.

\begin{theorem}\label{thm320} Let $\aa$ be a unital $C^*$-algebra and consider $a \in \aa$.
Let  $\suc{a_n}\subset \core{\aa}=\dcore{\aa}$ be such that $\suc{a_n}$ converges to $a$. The following
statements are equivalent.
\begin{enumerate}[{\rm (i)}]
\item  The element $a\in\core{\aa}$ and $\suc{\core{a}_n}$ converges to $\core{a}$.
\item  The element $a\in\dcore{\aa}$ and $\suc{\dcore{a_n}}$ converges to $\dcore{a}$.
\item  The element $a\in\aa^\#$ and $\suc{a_n^\#}$ converges to $a^\#$.
\item  The element $a\in\core{\aa}$ and $\suc{\core{a_n}}$ is a bounded sequence.
\item  The element $a\in\dcore{\aa}$ and $\suc{\dcore{a_n}}$ is a bounded sequence.
\item The element $a\in\aa^\dag$,  $\suc{a_n^\dag}$ converges to $a^\dag$, and  $\suc{\core{a_n}a_n}$ is a bounded sequence.
\item The element $a\in\aa^\dag$, $\suc{a_n^\dag}$ converges to $a^\dag$, and $\suc{a_n\dcore{a_n}}$ is a bounded sequence.
\item  The element $a\in\aa^\dag$, $\suc{a_n^\dag}$ converges to $a^\dag$, and there exists $\psi\in[0, \frac{\pi}{2})$ such that 
$\psi_n=\psi_{a_n}\le \psi$ for all $n\in\ene$.
\end{enumerate}
\end{theorem}
\begin{proof}  Note  that according to Theorem \ref{theo310},  $\suc{a_n}\subset \core{\aa}\cap\dcore{\aa}\cap\aa^\#\cap\aa^\dag$.

First the equivalence between statements (i) and (iii) will be proved.
Suppose that statement (i) holds. Then according to Theorem \ref{theo310}, $a\in\aa^\#$. In addition, $\suc{\core{a}_n a_n}$ converges to $\core{a}a$.
However, according to  \cite[Remark 2.17]{rdd},
$a^\#a=\core{a}a$, and for each $n\in\ene$, $a_n^\#a_n=\core{a}_na_n$. Consequently,
 $\suc{a_n^\#a_n}$ converges to $a^\#a$, which according to \cite[Theorem 2.4]{kr1}, implies that
 $\suc{a_n^\#}$ converges to $a^\#$. 

Suppose that statement (iii) holds. Note that according to \cite[Theorem 6]{hm1}, $a\in\aa^\dag$. In particular, according to Theorem \ref{theo310},
$a\in \core{\aa}$. Moreover, according to \cite[Corollary 2.1 (ii)]{bc} and 
\cite[Equation (2.1)]{cbl2},
$$
\| a_n^\dag \| = \| (a_n a_n^\dag +a_n^\dag a_n - \uno) a_n^\#
(a_n a_n^\dag +a_n^\dag a_n - \uno) \| \leq
\|a_n a_n^\dag +a_n^\dag a_n - \uno\|^2 \| a_n^\#\| \leq \| a_n^\# \|.
$$
\noindent Consequently,  $\suc{a_n^\dag}$ is a bounded sequence. 

Now two cases need to be considered. If $a=0$, then $\core{a}=0$. However, according to \cite[Theorem 2.19]{rdd}, $\core{a_n}=a_n^\#a_na_n^\dag$. Since  $\suc{a_n}$ converges to $0$ and 
$\suc{a_n^\#}$ and $\suc{a_n^\dag}$ are bounded sequences, $\suc{\core{a_n}}$ converges to $0$. 

Now suppose that $a\neq 0$. Since $\suc{a_n}$ converges to $a$, there exists and $n_0\in\ene$ such that $a_n\neq 0$, $n\ge n_0$. Without loss of generality, it is possible to assume that
$\suc{a_n}\subset\core{\aa}\setminus\{ 0\}$. Thus, according to \cite[Theorem 1.6]{kol}, $\suc{a_n^\dag}$ converges to $a^\dag$. However, according again to   \cite[Theorem 2.19]{rdd},
$\core{a}=a^\#aa^\dag$ and for each $n\in\ene$,  $\core{a}_n=a_n^\# a_n a_n^\dag$. 
Therefore, $\suc{\core{a}_n}$ converges to $\core{a}$.

To prove the equivalence between statements (ii) and (iii), apply a similar argument to the one used to prove the equivalence
between statements (i) and (iii). In particular, use the following identities, which holds for $b\in\dcore{\aa}$: ($\alpha$) $b^\# b=b\dcore{b}$
(\hspace{-1pt}\cite[Remark 2.17]{rdd}); ($\beta$) $\dcore{b}=b^\dag bb^\#$ (\hspace{-1pt}\cite[Theorem 2.19]{rdd}).

It is evident that statement (i) implies statement (iv). Now suppose that statement (iv) holds. It will be proved that statement (iii) holds.
According to Theorem \ref{theo310}, $a\in\aa^\#$.
In addition, according to \cite[Theorem 2.19]{rdd}, for each $n\in\ene$, $a_n^\#=(\core{a_n})^2a_n$. In particular, $\suc{a_n^\#}$ is a bounded sequence.
Consequently, according to \cite[Theorem 2.4]{kr1}, $\suc{a_n^\#}$ converges to $a^\#$, equivalently, statement (iii) holds.

The equivalence between statements (ii) and (v) can be proved applying a similar argument to the one used to prove the equivalence between statements (i) and (iv),
using in particular Lemma \ref{lema2} (iv).

Next it will be proved that statement (iv) implies statement (vi). Suppose then that statement (iv) holds. Then, $\suc{\core{a_n}a_n}$ is a bounded sequence. In addition, 
according to Theorem \ref{theo310}, $a\in\aa^\dag$. Now two cases need to be considered. Suppose first that $a=0$. Since statement (iv) and (v) are equivalent,
$\suc{\dcore{a_n}}$ is a bounded sequence. According to \cite[Theorem 2.19]{rdd}, for each $n\in\ene$, $a_n^\dag=\dcore{a_n}a_n\core{a_n}$. Since $\suc{\core{a_n}}$
and $\suc{\dcore{a_n}}$ are bounded sequences, $\suc{a_n^\dag}$ converges to $0=a^\dag$. 

If $a\neq 0$, as when it was proved that statement (iii) implies statement (i), it is possible to assume that $\suc{a_n}\subset\aa\setminus\{ 0\}$. According to Lemma \ref{lema2} (ii),
$$
\parallel a_n^\dag\parallel\le\parallel a_na_n^\dag +a_n^\dag a_n-\uno\parallel\parallel\core{a_n}\parallel\le 3\parallel\core{a_n}\parallel.
$$
\noindent In particular, $\suc{a_n^\dag}$ is a bounded sequence. However, according to \cite[Theorem 1.6]{kol}, $\suc{a_n^\dag}$ converges to $a^\dag$.

Suppose that statement (vi) holds. It will be proved that statement (vi) implies statement (iv). According to \cite[Theorem 2.19]{rdd}, for each $n\in\ene$, 
$\core{a_n}=\core{a_n}a_na_n^\dag$. Since $\suc{a_n^\dag}$ and $\suc{\core{a_n}a_n}$ are bounded sequences, $\suc{\core{a_n}}$ is a bounded sequence.

Note that according to Lemma \ref{lema2} (iii), for each $n\in\ene$, $b_n=a_na_n^\dag +a_n^\dag a_n-\uno$ is invertible and $b_n^{-1}=\core{a_n}a_n+(\core{a_n}a_n)^*-\uno$.
In addition, the sequence $\suc{b_n^{-1}}$ is bounded. In fact, according to \cite[Lemma 2.3]{kr}, $\parallel b_n^{-1}\parallel=\parallel\core{a_n}a_n\parallel$. 
Now,
since $\suc{b_n}$ converges to $b=aa^\dag +a^\dag a-\uno$, the element $b$ is invertible, which in view of Lemma \ref{lema2} (iii), is equivalent to 
 $a\in\core{\aa}$.  

According to \cite[Theorem 2.19]{rdd}, for each $n\in\ene$, $\core{a_n}a_n=a_n\dcore{a_n}$. Thus, statement (vii) is an equivalent formulation of statement (vi).

Finally, statements (vi) and (viii) will be proved to be equivalent. In fact, note that if $a_n=0$, then $\psi_n=0$. In addition, according to Lemma \ref{lema2} (iii), $a_na_n^\dag+a_n^\dag a_n-\uno$ is invertible, and  when $a_n\neq0$, according to  
Lemma \ref{lema2} (iii), \cite[Theorem 2.4 (iii)]{cbl2} and \cite[Lemma 2.3]{kr},
$$
\frac{1}{\cos \psi_n}=\parallel (a_na_n^\dag +a_n^\dag a_n-\uno)^{-1}\parallel= \parallel \core{a_n}a_n+(\core{a_n}a_n)^*-\uno\parallel= \parallel \core{a_n}a_n\parallel.
$$
\noindent In particular, $\suc{\core{a_n}a_n}$ is bounded if and only if there exists $\psi\in[0, \frac{\pi}{2})$ such that $\psi_n\le\psi$ for all $n\in\ene$.
\end{proof}

Theorem \ref{thm320} shows that the continuity of the group inverse and of the Moore-Penrose inverse are central for the continuity of the core inverse and the dual core inverse.
To learn more on the continuity of the group inverse and the Moore-Penrose inverse, see for example \cite{cbl, cbl2, kr1, V, rs} and \cite{hm, kol, mb, rs, V2}, respectively,
see also \cite[Chapter 4]{dr}.

Observe that the conditions in statement (vi) of Theorem \ref{thm320}, ($\alpha$) $a \in \aa^\dag$, $a_n^\dag \to a^\dagger$, and 
($\beta$) $\{ \core{a_n} a_n \}$ is a bounded sequence,  are independent from each other, as the following two examples
show.

\begin{example}\label{example1}
{\rm 
Consider $\ce$ as a $C^*$-algebra. Let $a_n = 1/n$ 
and $a=0$. It is evident that $a_n \to a$, $a_n^\dag = n$, and $\suc{ a_n^\dag }$ does not
converge to $a^\dag =0$. However, it should be clear that $\core{a}_n=n$. Therefore, 
$\core{a}_n a_n = 1$, and thus, $\suc{ \core{a_n} a_n }$ is a bounded sequence.
}
\end{example}

\begin{example}\label{ejemplo2}
{\rm 
Consider the set of $2 \times 2$ complex matrices as a $C^*$-algebra. Take
the conjugate transpose of the matrix as the involution on this matrix. Let $\suc{\psi_n}$ be 
a sequence in $(0,\pi/2)$ such that $\psi_n \to \pi/2$ and let 
\begin{align*}
& &&A_n = \left[ \begin{array}{cc} \cos \psi_n & \sin \psi_n  \\ 0 & 0 \end{array} \right],& 
&A = \left[ \begin{array}{cc} 0 & 1 \\ 0 & 0 \end{array} \right]. &\\
\end{align*}
It is simple prove  that 
\begin{align*}
&A_n^\dag = \left[ \begin{array}{cc} \cos \psi_n & 0 \\ \sin \psi_n & 0 \end{array} \right], &
&A^\dag = \left[ \begin{array}{cc} 0 & 0 \\ 1 & 0 \end{array} \right],&
&\core{A}_n = \left[ \begin{array}{cc} 1/\cos \psi_n & 0 \\ 0 & 0 \end{array} \right].&
\end{align*}
Therefore, $\suc{A_n^\dag}$ converges to $A^\dag$ and 
\begin{align*}
& &&\core{A}_n A_n = \left[ \begin{array}{cc} 1 & \tan \psi_n \\ 0 & 0 \end{array} \right],&\\
\end{align*}
which shows that $\suc{ \core{A}_n A_n }$ is not bounded. Note also that $\suc{\core{A_n}}$ is not a convergent sequence.
}
\end{example}

\indent Observe also that if $\aa$ is a unital $C^*$-algebra and $\suc{a_n}\subset\core{\aa}$ is such that $\suc{a_n}$ converges to $a\in\aa$,
Example \ref{example1} also shows that the condition $\suc{\core{a_n}a_n}$ is a convergent sequence does not implies that 
$\suc{\core{a_n}}$ is convergent.

\indent It is worth noticing that Example \ref{ejemplo2} also proves that $\core{\aa}=\dcore{\aa}$ is not in general a closed set. 
In fact, using the same notation as in Example \ref{ejemplo2}, $\suc{A_n}\subset\core{\aa}$, $\suc{A_n}$ converges to $A$ but
$A\notin\core{\aa}$ ($A^2=0$, $\rk(A^2)=0\neq 1=\rk(A)$, i.e., $A$ is not group invertible).

\indent Next  an extension of \cite[Theorem 2.7]{cbl2} will be derived from Theorem \ref{thm320}.

\begin{corollary}\label{cor3900} Let $\aa$ be a unital $C^*$-algebra and consider $a\in\aa$. Suppose that the sequence $\suc{a_n}\subset\aa^\#$ is such that
$\suc{a_n}$ converges to $a$. Then, the  following statements are equivalent.
\begin{enumerate}[{\rm (i)}]
\item  The element $a\in\aa^\#$ and $\suc{a_n^\#}$ converges to $a^\#$.
\item The sequence $\suc{a_n^\#}$ is bounded.
\item The element $a\in\aa^\dag$, $\suc{a_n^\dag}$ converges to $a^\dag$, and the sequence $\suc{a_n^\# a_n}$ is bounded. 
\item  The element $a\in\aa^\dag$, $\suc{a_n^\dag}$ converges to $a^\dag$, and there exists $\psi\in[0, \frac{\pi}{2})$ such that 
$\psi_n=\psi_{a_n}\le \psi$ for all $n\in\ene$.
\end{enumerate}
\end{corollary}
\begin{proof} Statement (ii) is a consequence of statement (i). 

Suppose that statement (ii) holds. Then, $\suc{a_n^\# a_n}$ is a bounded  sequence. To prove that $a\in\aa^\dag$ and $\suc{a_n^\dag}$ converges to $a^\dag$, proceed as
in the corresponding part of the proof of  \cite[Theorem 2.7]{cbl2} (see statement (ii) implies statement (iii) in  \cite[Theorem 2.7]{cbl2}).

Suppose that statement (iii) holds. First note that if $a_n=0$, then $\psi_n=0$. In addition, according to  \cite[Theorem 2.5]{cbl2}, if $a_n\neq 0$, then,
$$
\|a_na_n^\# \|=\frac{1}{\cos \psi_{a_n}}.
$$
\noindent Therefore, the sequence $\suc{a_n^\# a_n}$ is bounded
if and only if there exists $\psi\in[0, \frac{\pi}{2})$ such that $\psi_n=\psi_{a_n}\le \psi$ for all $n\in\ene$.

To prove that statement (iv) implies statement (i), apply Theorem \ref{thm320} (equivalence between statements (iii) and (viii)).
\end{proof}

In Theorem \ref{thm320} and Corollary \ref{cor3900} the general case has been presented for sake of completeness. However, the case $a=0$ is particular and it  deserves to be studied.
Recall that given a unital $C^*$-algebra $\aa$, if $\suc{a_n}\subset\aa^{-1}$ is such that $\suc{a_n}$ converges to 0, then the sequence 
$\suc{a_n^{-1}}$ is unbounded. Next the case of a sequence $\suc{a_n}\subset\aa^\#=\core{\aa}=\dcore{\aa}\subseteq \aa^\dag$ such that  it converges to $0$ will be studied.
In first place the Moore-Penrose inverse will be considered.

\begin{remark}\label{rem3950}\rm Let $\aa$ be a unital $C^*$-algebra and consider $a\in\aa^\dag$ and $\suc{a_n}\subset\aa^\dag$ such that $\suc{a_n}$ converges to $a$.
Recall that according to \cite[Theorem 1.6]{kol},
the following statements are equivalent.
\begin{enumerate}[{\rm (i)}]
\item The sequence $\suc{a_n^\dag}$ converges to $a^\dag$.
\item The sequence $\suc{a_n a_n^\dag}$ converges to $aa^\dag$.
\item The sequence $\suc{a_n^\dag a_n}$ converges to $a^\dag a$.
\item The sequence $\suc{a_n^\dag}$ is bounded.

\hspace*{\dimexpr\linewidth-\textwidth\relax}{\noindent Now when $a=0$, according to \cite[Theorem 1.3]{kol}, the following equivalence holds:}

\item A  necessary and sufficient condition for $\suc{a_n^\dag}$ to converge to 0 is that the sequence $\suc{a_n^\dag}$ is bounded.

 \hspace*{\dimexpr\linewidth-\textwidth\relax}
\begin{minipage}[t]{\textwidth}
\noindent However, concerning the convergence of $\suc{a_n a_n^\dag}$, note that given $n\in\ene$, since $a_n a_n^\dag$ is
a self-adjoint idempotent, if $\parallel a_n a_n^\dag\parallel < 1$, then $a_n a_n^\dag =0$, which 
implies that $a_n=0$; \hskip.1truecm a similar result can be derived for the convergence of $\suc{a_n^\dag a_n}$. Consequently, the following statements are equivalent.
\end{minipage}

\item The sequence $\suc{a_n^\dag}$ converges to 0.
\item There exists $n_0\in\ene$ such that for $n\ge n_0$, $a_n=0$.

\hspace*{\dimexpr\linewidth-\textwidth\relax}{\noindent Therefore, according to statements (v)-(vii), given $\suc{a_n}\subset\aa^\dag$ such that $\suc{a_n}$ converges to $0$, there are only two possibilities.}

\item There exists $n_0\in\ene$ such that for $n\ge n_0$, $a_n=0$; or
\item the sequence $\suc{a_n^\dag}$ is unbounded.
\end{enumerate}
\end{remark}

In the following proposition, sequences of group invertible or (dual) core invertible elements that converge to 0 will be studied.

\begin{proposition}\label{prop3960} Let $\aa$ be a unital $C^*$-algebra and consider a sequence $\suc{a_n}\subset \aa^\#=\core{\aa}=\dcore{\aa}$.
such that  $\suc{a_n}$ converges to 0. The following statements are equivalent.
\begin{enumerate}[{\rm (i)}]
\item The sequence $\suc{a_n^\#}$  converges to 0.
\item The sequence $\suc{\core{a_n}}$  converges to 0.
\item The sequence $\suc{\dcore{a_n}}$  converges to 0.
\item The sequence $\suc{a_n^\#}$ is bounded.
\item There exists $n_0\in\ene$ such that for $n\ge n_0$, $a_n=0$.

\hspace*{\dimexpr\linewidth-\textwidth\relax}{\noindent In addition, there exist only two possibilities for the sequence $\suc{a_n}$.}

\item There exists $n_0\in\ene$ such that for $n\ge n_0$, $a_n=0$; or
\item the sequence $\suc{a_n^\#}$ is unbounded.

\hspace*{\dimexpr\linewidth-\textwidth\relax}{\noindent Moreover, statement {\rm (vii)} is equivalent to the following two statements.}

\item the sequence $\suc{\core{a_n}}$ is unbounded.
\item the sequence $\suc{\dcore{a_n}}$ is unbounded.
\end{enumerate}
\end{proposition}
\begin{proof} According to Theorem \ref{thm320}, statements (i)-(iii) are equivalent.

It is evident that statement (i) implies statement (iv).

Suppose that statement (iv) holds. According to \cite[Theorem 6]{hm1}, $\suc{a_n}\subset \aa^\#\subset \aa^\dag$. 
In addition according to \cite[Corollary 2.1 (ii)]{bc},
$$
\parallel a_n^\dag\parallel\le \parallel a_n^\#\parallel \parallel a_na_n^\dag+a_n^\dag a_n-\uno\parallel^2\le 9 \parallel a_n^\#\parallel. 
$$
\noindent In particular, the sequence $\suc{a_n^\dag}$ is bounded. Thus, according to Remark \ref{rem3950} (v), $\suc{a_n^\dag}$ converges to 0. However, according to Remark \ref{rem3950} (vi)-(vii),
statement (v) holds.

It is evident that statement (v) implies statement (i).

Statements (vi) and (vii) can be derived from what has been proved.

According to Theorem  \ref{thm320}, statements (vii)-(ix) are equivalent.
\end{proof}
To prove the second characterization of this section some preparation is needed.

\begin{remark}\label{nota2}\rm Let $\aa$ be a unital $C^*$-algebra and consider $a \in \core{\aa}=\dcore{\aa}$. If $L_a\colon \aa\to \aa$ and $R_a\colon \aa\to \aa$ are the left and the right multiplication operators defined by $a$, i.e., for $x\in\aa$, $L_a(x)=ax$,  $R_a(x)=xa$, respectively, then according to \cite[Theorem 2.14]{rdd},
\begin{align*}
& &&L_{\core{a}}L_a L_{\core{a}}=L_{\core{a}},& &R_{\core{a}}R_a R_{\core{a}}=R_{\core{a}}.&&\\
\end{align*}

Note also that according to Definition \ref{df1},
\begin{align*}
& &&\rr(L_{\core{a}})=a\aa,&  &\nn(L_{\core{a}})=(a^*)^\circ.&\\
& &&\rr(R_{\core{a}})=\aa a^*,&  &\nn(R_{\core{a}})={}^\circ a.&\\
\end{align*}
Therefore, $L_{\core{a}} = (L_a)^{(2)}_{a \aa, (a^*)^\circ}$ and $R_{\core{a}} = (R_a)^{(2)}_{ \aa a^*, {}^\circ a}$.

\noindent In addition, since $L_{a \core{a}}=L_a L_{\core{a}}$, $R_{a \core{a}}=R_{\core{a}}R_a\in\ll (\aa)$ are idempotents, observe that according to
Definition \ref{df1} and \cite[Theorem 2.14]{rdd},
\begin{align*}
& &&\rr(L_{a \core{a}})=a \aa,&  &\rr(R_{a \core{a}})=\aa a^*,&\\
& &&\nn(L_{a \core{a}})=(a^*)^\circ ,& &\nn(R_{a \core{a}})={}^\circ a.\\
\end{align*}

 Similar arguments prove the following facts: $L_{\dcore{a}} = (L_a)^{(2)}_{a^* \aa, a^\circ}$, $R_{\dcore{a}} = (R_a)^{(2)}_{ \aa a, {}^\circ(a^*)}$ and
\begin{align*}
& &&\rr(L_{\dcore{a}a})=a^* \aa,&  &\rr(R_{\dcore{a}a})=\aa a,&\\
& &&\nn(L_{\dcore{a}a})=a^\circ ,& &\nn(R_{\dcore{a}a})={}^\circ (a^*).\\
\end{align*}
\end{remark}

Next follows the second characterization of the continuity of the (dual) core inverse.
In this case, the notion of the gap between subspaces will be used.

\begin{theorem} \label{th1}
Let $\aa$ be a unital $C^*$-algebra and consider $a \in \core{\aa}=\dcore{\aa}$, $a\neq 0$.
Consider a sequence $\suc{a_n} \subset \core{\aa}=\dcore{\aa}$ such that $\suc{a_n}$ converges to $a$. The following
statements are equivalent.
\begin{enumerate}[{\rm (i)}]

\item $\suc{\core{a}_n}$ converges to $\core{a}$.

\item $\suc{a_n \core{a}_n}$ converges to $a \core{a}$.
	
\item $\suc{\gap{a_n \aa}{a\aa}}$ and $\suc{\gap{(a_n^*)^\circ}{(a^*)^\circ}}$
	converge to $0$.
\item $\suc{\gap{\aa a_n^*}{\aa a^*}}$ and $\suc{\gap{{}^\circ a_n}{{}^\circ a}}$
	converge to $0$.

\item $\suc{\dcore{a_n}}$ converges to $\dcore{a}$.

\item $\suc{\dcore{a_n} a_n}$ converges to $\dcore{a} a$.
	
\item $\suc{\gap{a_n^* \aa}{a^*\aa}}$ and $\suc{\gap{a_n^\circ}{a^\circ}}$
	converge to $0$.
\item $\suc{\gap{\aa a_n}{\aa a}}$ and $\suc{\gap{{}^\circ (a_n^*)}{{}^\circ (a^*)}}$
	converge to $0$.

\end{enumerate}
\end{theorem}
\begin{proof}
It is evident that statement (i) implies statement (ii). Suppose that statement (ii) holds.
According to Remark \ref{nota2}, 
$a \aa = \rr(L_{a \core{a}})$, $(a^*)^\circ =\nn(L_{a \core{a}})$, $a_n \aa = \rr(L_{a_n \core{a}_n})$
and 
$(a_n^*)^\circ =\nn(L_{a_n \core{a}_n})$ ($n \in \ene$).
However, according to \cite[Lemma 3.3]{kr1}, statement (iii) holds.

Suppose that statement (iii) holds. Recall that  according to 
Remark~\ref{nota2}, 
\begin{align*}
& & &L_{\core{a}} = (L_a)^{(2)}_{a \aa, (a^*)^\circ},& &L_{\core{a}_n} = (L_{a_n})^{(2)}_{a_n \aa, (a_n^*)^\circ},&\\
\end{align*}
\noindent for each $n \in \ene$. 
Let $\kappa = \| L_a \| \| L_{\core{a}} \| = \| a \| \| \core{a} \|$ and
consider $n_0 \in \ene$ such that for all $n \geq n_0$,
\begin{align*}
& &&r_n =\gap
{\nn \left( (L_{a_n})^{(2)}_{a_n \aa, (a_n^*)^\circ}\right)}
{\nn \left( (L_a)^{(2)}_{a \aa, (a^*)^\circ} \right)} =\gap{(a_n^*)^\circ}{(a^*)^\circ}< \frac{1}{3+\kappa},&\\
& &&s_n = 
\gap
{\rr \left( (L_{a_n})^{(2)}_{a_n \aa, (a_n^*)^\circ}\right)}
{\rr \left( (L_a)^{(2)}_{a \aa, (a^*)^\circ} \right)} =\gap{a_n \aa}{a \aa} < \frac{1}{(1+\kappa)^2},&\\
\end{align*}
and
\begin{align*}
& &&t_n = \| L_{\core{a}} \| \| L_a - L_{a_n} \| = \| \core{a} \| \| a - a_n \| < 
\frac{2 \kappa}{(1+\kappa)(4+\kappa)}.&\\
\end{align*}
Thus, according to \cite[Theorem~3.5]{dx},
\begin{align*}
& &&\left\| \core{a}_n - \core{a} \right\| =
\left\| L_{\core{a}_n} - L_{\core{a}} \right\| \leq 
\frac
{(1+\kappa)(s_n+r_n)+(1+r_n)t_n}
{1-(1+\kappa)s_n-\kappa r_n - (1+r_n)t_n}\| \core{a} \|, &\\
\end{align*}
which implies statement (i).

Statements (i), (ii) and (iv) are equivalent. To prove this fact, apply a similar argument to the one used to prove the
equivalence among statements (i), (ii) and (iii), using in particular 
$R_{\core{a}}=(R_a)^{(2)}_{\aa a^*, {}^\circ a}$, $R_{\core{a}_n}=(R_{a_n})^{(2)}_{\aa a_n^*, {}^\circ a_n}$, 
$R_{a\core{a}}$ and $R_{a_n\core{a}_n}$
instead of the respectively left multiplication operators (Remark~\ref{nota2}, $n\in\ene$).

Statements (i) and (v) are equivalent (Theorem \ref{thm320}).

To prove the equivalence among statements (v) and (viii), apply a similar argument to the one used to prove that
statements (i)-(iv) are equivalent, using in particular Remark~\ref{nota2} and \cite[Theorem~3.5]{dx}.
\end{proof}

Next, some bounds for $\| \core{a}_n - \core{a} \|$ will be proved, when 
$\suc{a_n}\subset\aa$ converges to $a\in\aa$ in a $C^*$-algebra $\aa$. Before, a technical lemma is presented.

\begin{lemma}\label{lema4}
Let $\aa$ be a unital $C^*$-algebra and let $a,b \in \core{\aa}=\dcore{\aa}$. Then
\begin{align*}
&{\rm (i)}&&\core{b}-\core{a} = \core{b}b(b^\dag -a^\dag)(\uno-a \core{a}) + \core{b}(a-b)\core{a} +
(\uno - \core{b}b)(b-a)a^\dag\core{a}. &\\
&{\rm (ii)}& &\dcore{b}-\dcore{a} = (\uno- \dcore{a}a) (b^\dag -a^\dag)b\dcore{b} + \dcore{a}(a-b)\dcore{b}+
\dcore{a}a^\dag(b-a)(\uno - b\dcore{b}).
\end{align*}
\end{lemma}
\begin{proof} To prove statement (i), recall that since $a$ and $b$ are core invertible, $a$ and $b$ are Moore-Penrose invertible (Theorem \ref{theo310}). 
In addition, according to \cite[Theorem 3.1]{14}, $b=\core{b}b^2$. Thus, according to Lemma \ref{lema2} (i), 
\begin{align*}
& &&(\uno - \core{b}b)(b-a)a^\dag\core{a} = -(\uno - \core{b}b)aa^\dag\core{a}
= -(\uno - \core{b}b)\core{a} = \core{b}b\core{a}-\core{a}.&\\
\end{align*}

Now, according to \cite[Theorem 2.19]{rdd},  $\core{b}=\core{b}bb^\dag$. 
In addition, $a^*a\core{a} = a^* (a \core{a})^* = (a\core{a}a)^* = a^*$, i.e., 
$a^*(\uno-a \core{a})=0$. Moreover, since
$a^\dag = a^\dag a a^\dagger = a^\dag (aa^\dag)^* = a^\dag (a^\dagger)^* a^*$, $a^\dag (\uno-a\core{a})=0$. Therefore,
\begin{align*}
& &&\core{b}b(b^\dag -a^\dag)(\uno-a \core{a}) =
\core{b}bb^\dag (\uno-a \core{a}) = \core{b}(\uno-a \core{a})
= \core{b}-\core{b}a\core{a}.&\\
\end{align*}

As a result,
\begin{align*}
& &&\core{b}-\core{a} & &= 
\core{b} - \core{b}a\core{a} + \core{b}a\core{a} - 
\core{b}b \core{a} + \core{b}b \core{a} - \core{a} \\
& & & && = \core{b}b(b^\dag -a^\dag)(\uno-a \core{a}) + \core{b}(a-b)\core{a} +
(\uno - \core{b}b)(b-a)a^\dag\core{a}. &\\
\end{align*} 

To prove statement (ii), use that $\dcore{x}=(\core{(x^*)})^*$ ($x\in\aa$), and apply statement (i).
\end{proof}

Next the aforementioned bounds will be given.

\begin{theorem} \label{th_des}
Let $\aa$ be a unital $C^*$-algebra and consider $a\in\core{\aa}=\dcore{\aa}$. The following statements holds.\par
\begin{enumerate}[{\rm (i)}]
\item If $b\in\core{\aa}=\dcore{\aa}$, $b\neq 0$, then
\begin{align*}
& & &\| \core{b} - \core{a} \| \leq 
\frac{\| b^\dag -a^\dag \|}{\cos \psi_b} + 
\left[ \| \core{b} \| + \frac{\| a^\dag\|}{\cos \psi_b} \right] 
 \| \core{a} \|\| a-b\|.&\\
\end{align*}
\item In addition,
\begin{align*}
& &&\| \dcore{b} - \dcore{a} \| \leq 
\frac{\| b^\dag -a^\dag \|}{\cos \psi_b} + 
\left[ \| \dcore{b} \| + \frac{\| a^\dag\|}{\cos \psi_b} \right] 
 \| \dcore{a} \|\| a-b\|.&\\
\end{align*}
\item If also $a\neq 0$, then 
\begin{align*}
& &&\| \dcore{b} - \dcore{a} \| =\| \core{b} - \core{a} \| \leq 
\frac{\| b^\dag -a^\dag \|}{\cos \psi_b} + 
\frac{\| a^\dag \| \left(\| b^\dag \|  + \| a^\dag \| \right)}{\cos \psi_a \cos \psi_b}
\| a-b \|.&\\
\end{align*}
\item In particular, if $a\in\core{\aa}=\dcore{\aa}$, $a\neq 0$, and $\suc{a_n}\subset \core{\aa}=\dcore{\aa}$, $a_n\neq 0$ for all $n\in\ene$, then 
\begin{align*}
& &&\| \dcore{a_n} - \dcore{a} \| =\| \core{a_n} - \core{a} \| \leq 
\frac{\| a_n^\dag -a^\dag \|}{\cos \psi_n} + 
\frac{\| a^\dag \| \left(\| a_n^\dag \|  + \| a^\dag \| \right)}{\cos \psi_a \cos \psi_n}
\| a-a_n \|,&\\
\end{align*}
\noindent where $\psi_n=\psi_{a_n}$.
\end{enumerate}
 \end{theorem}

\begin{proof} To prove statement (i), note that according to 
 \cite[Lemma 2.3]{kr}, Lemma \ref{lema2} (iii) and \cite[Theorem 2.4 (iii)]{cbl2}
$$
\| \uno- \core{b}b \| = \| \core{b} b \|
= \| (\core{b}b) + (\core{b} b)^* - \uno\| = \| 
(b b^\dag + b^\dag b - \uno)^{-1} \| = \frac{1}{\cos \psi_b}.
$$
Observe that $\uno-a \core{a}$ is a self-adjoint idempotent, 
Hence $\| \uno - a \core{a} \| = 1$, and according to Lemma \ref{lema4}, 
\begin{align*}
\| \core{b}  -\core{a} \| 
 &\leq \| \core{b}b \| \| b^\dag -a^\dag \| \| \uno-a \core{a}\| + 
\left[ \| \core{b} \|  \| \core{a} \| + \| \uno - \core{b}b \|  \| a^\dag\core{a} \|
\right] \| a-b \|  \\
& = \frac{\| b^\dag -a^\dag \|}{\cos \psi_b} + 
\left[ \| \core{b} \|  \| \core{a} \| + \frac{\| a^\dag\core{a} \|}{\cos \psi_b}
\right] \| a-b \|. \\
& \le \frac{\| b^\dag -a^\dag \|}{\cos \psi_b} + 
\left[ \| \core{b} \|   + \frac{\| a^\dag \|}{\cos \psi_b}
\right] \| \core{a} \|\| a-b \|. \\
\end{align*} 

Statement (ii) can be derived from statement (i). In fact, recall that given $x\in\core{\aa}=\dcore{\aa}$,
$\dcore{x}=(\core{(x^*)})^*$. Moreover, if $x\in \aa^\dag\setminus\{ 0\}$, then note that $\psi_{x^*}=\psi_{x^\dag}=\psi_x$.
Now apply statement (i) to $a^*$ and $b^*$. 

\indent Now observe that if $a=0$ in statement (i), then $\| \core{b}\|\le\frac{\|b^\dag\|}{\cos \psi_b}$.
Thus, if $a\neq 0$, then $\| \core{a}\|\le\frac{\|a^\dag\|}{\cos \psi_a}$. To prove statement (iii) for the core inverse, apply these inequalities to 
statement (i). To prove statement (iii) for the dual core inverse, proceed as in the proof of statement (ii).

\indent Statement (iv) can be derived from statement (iii). 
\end{proof}

\begin{remark} \rm As it was used in the proof of Theorem \ref{th_des}, given $a\in\core{\aa}=\dcore{\aa}$, $a\neq 0$,
Theorem   \ref{th_des} (i) (respectively Theorem   \ref{th_des} (ii)) gives a relationship between the norm
of $\core{a}$ (respectively of $\dcore{a}$) and the norm of the $a^\dag$: $\| \core{a}\|\le\frac{\|a^\dag\|}{\cos \psi_a}$
(respectively $\| \dcore{a}\|\le\frac{\|a^\dag\|}{\cos \psi_a}$). 

\noindent Moreover, under the same hypothesis of Theorem \ref{thm320}, when $a\neq 0$, Theorem  \ref{th_des} (iv)
gives an estimate of the convergence of $\suc{\core{a_n}}$ and $\suc{\dcore{a_n}}$ to $\core{a}$ and $\dcore{a}$, respectively. 
\end{remark}

\section{Continuity of (dual) core invertible Hilbert space operators}

Let $\hh$ be a  Hilbert space and consider $A\in\ll (\hh)$. The definition of core invertible Hilbert space operators was given in \cite[Definition 3.2]{rdd2}. In fact,
$A\in\ll (\hh)$ is said to be core invertible, if there exists $X\in\ll (\hh)$ such that 
$$
A=AXA,\hskip.3truecm \rr(X)=\rr(A),\hskip.3truecm \nn (X)=\nn (A^*). 
$$
\noindent Thus, when $A\in\ll (\hh)$, two definitions of the core inverse of $A$ has been given: as an element of the $C^*$-algebra $\ll (\hh)$ and as 
Hilbert space operator. However, as the following proposition shows, both definitions coincide in the Hilbert space context.

\begin{proposition}\label{proposition410} Let $\hh$ be a Hilbert space and consider $A\in\ll (\hh )$. The following statements are equivalent.\par
\begin{enumerate}[{\rm (i)}] 
\item The core inverse of $A$ exists.
\item  There exists an operator $X\in \ll (\hh )$ such that $AXA=A$, $\rr(X)=\rr(A)$ and $\nn(X)=\nn(A^*)$. 
\end{enumerate}
\noindent Moreover, in this case $X=\core{A}=A^{(2)}_{\rr(A), \nn(A^*)}$.
\end{proposition}
\begin{proof} Suppose that $\core{A}$ exists. Then, $A=A\core{A}A$ and  there are operator $S$, $T$, $U$, $V\in\ll (\hh )$ such that 
$$
\core{A}=AS,  \hskip.2truecm A=\core{A}T,  \hskip.2truecm\core{A}=UA^*,  \hskip.2truecmA^*=V\core{A}.
$$
\noindent In particular, $\rr(\core{A})=\rr(A)$ and $\nn(\core{A})=\nn(A^*)$.\par
\indent Now suppose that statement (ii) holds. Then, there exists $X\in\ll (\hh )$ such that $\rr(X)=\rr(A)$. According to \cite[Theorem 1]{Do},
there are $L$, $K\in\ll (\hh )$ such that $A=XL$ and $X=AK$. In particular, $X\ll (\hh )= A\ll (\hh )$.  In addition, since $\rr (A)$ is closed, $\rr(X)$  
is closed, which is equivalent to the fact that $X$ is regular. 
Now since $A^*$ is regular, according to \cite[Remark 6]{b}, there exist operators $M$, $N\in \ll (\hh )$ such that $X=MA^*$ and $A^*=NX$. In particular, $\ll (\hh )X= \ll (\hh )A^*$.
Since $A=AXA$ and the core inverse is unique, when it exists (\hspace{-1pt}\cite[Theorem 2.14]{rdd}), $X=\core{A}$. Finally, since according again to \cite[Theorem 2.14]{rdd}, 
$\core{A}$ is an outer inverse, according to what has been proved, $\core{A}=A^{(2)}_{\rr(A), \nn(A^*)}$.
\end{proof}

\indent As for the core inverse case, a definition of dual core invertible Hilbert space operators was given in \cite[Definition 3.3]{rdd2}. In the following proposition
the equivalence between Definition  \ref{df2} and  \cite[Definition 3.3]{rdd2} will be considered.

\begin{proposition}\label{proposition420} Let $\hh$ be a Hilbert space and consider $A\in\ll (\hh )$. The following statements are equivalent.\par
\begin{enumerate}[{\rm (i)}] 
\item The dual core inverse of $A$ exists.
\item  There exists an operator $X\in \ll (\hh )$ such that $AXA=A$, $\rr(X)=\rr(A^*)$ and $\nn(X)=\nn(A)$. 
\end{enumerate}
\noindent Moreover, in this case $X=\dcore{A}=A^{(2)}_{\rr(A^*), \nn(A)}$.
\end{proposition}
\begin{proof} Apply a similar argument  to the one used in Proposition \ref{proposition410}.
\end{proof}

\indent Note that the relationship between the (dual) core inverse and the outer inverse with prescribed range and null space
for the case of square complex matrices was studied in \cite[Theorem 1.5]{r} (apply \cite[Theorem 4.4]{rdd}).

\indent Next the continuity of the (dual) core inverse will be characterized using the gap between subspaces.
The next theorem is the Hilbert space version of Theorem \ref{th1}

\begin{theorem}\label{thm430}Let $\hh$ be a Hilbert space and consider $A\in \ll (\hh)$, $A\neq 0$, such that
$A$ is (dual) core invertible. Suppose that there exists a sequence of 
operators $\suc{A_n}\subset \ll (\hh )$ such that
for each $n\in\ene$, $A_n$ is (dual) core invertible and  $\suc{A_n}$ converges to $A$. Then, the following statements are equivalent.
\begin{enumerate}[{\rm (i)}] 
\item The sequence $\suc{\core{A_n}}$ converges to $\core{A}$.
\item The sequence $\suc{\dcore{A_n}}$ converges to $\dcore{A}$.
\item The sequence $\suc{\core{A_n}A_n}$ converges to $\core{A}A$. 
\item The sequence $\suc{A_n\dcore{A_n}}$ converges to $\core{A}A$. 
\item The sequence $\suc{\gap{\rr(\core{A_n})}{\rr(\core{A})}}$ converges to 0.
\item  The sequence $\suc{\gap{\rr(A_n)}{\rr(A)}}$ converges to 0.
\item The sequence $\suc{\gap{\nn (\core{A_n})} {\nn (\core{A})}}$ converges to 0.
\item The sequence $\suc{\gap{\nn (A_n^*)}{\nn (A^*)}}$ converges to $0$.
\item The sequence $\suc{\gap{\rr(\dcore{A_n})}{\rr(\dcore{A})}}$ converges to 0.
\item  The sequence $\suc{\gap{\rr(A_n^*)}{\rr(A^*)}}$ converges to 0.
\item The sequence $\suc{\gap{\nn (\dcore{A_n})} {\nn (\dcore{A})}}$ converges to 0.
\item The sequence $\suc{\gap{\nn (A_n)}{\nn (A)}}$ converges to $0$.

\end{enumerate}
\end{theorem}

\begin{proof} First of all recall  that $\core{\ll (\hh)}=\dcore{\ll (\hh)}$ (Theorem \ref{theo310}).\par

Statements (i)-(iv) are equivalent (Theorem \ref{th1}). 
According to \cite[Lemma 3.3]{kr1}, statement (iii) implies statement (v) and according to Proposition \ref{proposition410}
and \cite[Chapter 4, Section 2, Subsection 3, Theorem 2.9]{kato}, Statements (v)-(viii) are equivalent.

\indent Now suppose that statement (vi) holds. Thus, according to what has been proved, the sequences $\suc{\gap{\rr(A_n)}{\rr(A)}}$
and $\suc{\gap{\nn(A_n^*)}{\nn(A^*)}}$ converge to $0$ (recall that according to  \cite[Chapter 4, Section 2, Subsection 3, Theorem 2.9]{kato}, 
$\gap{\rr(A_n)}{\rr(A})=\gap{(\nn(A_n^*)}{\nn(A^*)}$, $n\in \ene$). In addition, according to Proposition \ref{proposition410}, for each $n\in\ene$,
\begin{align*}
& &&\core{A}_n=(A_n)^{(2)}_{\rr(A_n), \nn(A_n^*)},&  &\core{A}=A^{(2)}_{\rr(A), \nn(A^*)}.&
\end{align*}

\noindent Let $\kappa=\parallel A\parallel\parallel \core{A}\parallel$ and consider $n_0\in\ene$ such that
for all $n\ge n_0$, 

\begin{align*}
w_n=&\gap{\nn((A_n)^{(2)}_{\rr(A_n), \nn(A_n^*)})}{\nn(A^{(2)}_{\rr(A), \nn(A^*)})}=\gap{(\nn(A_n^*)}{\nn(A^*)}\\
&=\gap{\rr(A_n)}{\rr(A)}=\gap{\rr((A_n)^{(2)}_{\rr(A_n), \nn(A_n^*)})}{\rr(A^{(2)}_{\rr(A), \nn(A^*)})}<\frac{1}{(3+\kappa)^2},\\
&z_n=\parallel \core{A}\parallel\parallel A-A_n\parallel<\frac{2\kappa}{(1+\kappa)(4+\kappa)}.\\
\end{align*}

\noindent Since $ \frac{1}{(3+\kappa)^2}\le\min\{\frac{1}{3+\kappa}, \frac{1}{(1+\kappa)^2}\}$,
according to \cite[Theorem 3.5]{dx},
$$
\parallel \core{A}_n-\core{A}\parallel \le\frac{2(1+\kappa)w_n+(1+w_n)z_n}{1-(1+2\kappa) w_n-(1+w_n)z_n}\parallel \core{A}\parallel,
$$
which implies statement (i).

Now, according to \cite[Lemma 3.3]{kr1}, statement (iv) implies statement (xi) and according to Proposition \ref{proposition420}
and \cite[Chapter 4, Section 2, Subsection 3, Theorem 2.9]{kato}, Statements (ix)-(xii) are equivalent.

Suppose that statement (x) holds. Since then statement (xii) also holds, to prove that statement (ii) holds, it is enough to apply an argument similar to the one
used to prove that statement (vi) implies statement (i), interchanging in particular $A$ with $A^*$, $A_n$ with $A_n^*$, $\core{A}$ with $\dcore{A}$,  
$\core{A_n}$ with $\dcore{A_n}$, $(A_n)^{(2)}_{\rr(A_n), \nn(A_n^*)}$  with $(A_n)^{(2)}_{\rr(A_n^*), \nn(A_n)}$, $A^{(2)}_{\rr(A), \nn(A^*)}$ with $A^{(2)}_{\rr(A^*), \nn(A)}$,
and $\kappa$ with $\kappa'=\parallel A\parallel \parallel \dcore{A}\parallel$. 
\end{proof}

\indent Next the continuity of the (dual) core inverse will be studied in a particular case. To this end, two results from \cite{rs} need to be extended first.

\begin{proposition}\label{pro540} Let $\xx$ be a Banach space and consider $A\in \ll (\xx)$ such that
$A$ is group invertible and the codimension of $\rr(A)$ is finite. Suppose that there exists a sequence of 
operators $\suc{A_n}\subset \ll (\xx )$ such that
for each $n\in\ene$, $A_n$ is group invertible and  $\suc{A_n}$ converges to $A$. Then
the following statements are equivalent.
\begin{enumerate}[{\rm (i)}]
\item The sequence $\suc{A^\#_n}$ converges to $A^\#$.\par
\item For all sufficiently large $n\in\ene$, $\hbox{\rm codim }\rr(A_n)=\hbox{\rm codim }\rr(A)$.
\end{enumerate}
\end{proposition}
\begin{proof} Recall that $A\in\ll (\xx)$ is group invertible if and only if $A^*\in \ll(\xx^*)$ is group invertible.
A similar statement holds for each $A_n\in\ll (\xx)$ ($n\in\ene$). In addition, $\dim \nn(A^*)$ is finite and $\suc{A^*_n}\subset \ll (\xx^* )$
converges to $A^*$. Thus, according to \cite[Theorem 3]{rs}, statement (i) is equivalent to the fact that 
for all sufficiently large $n\in\ene$, $\dim \nn(A^*_n)=\dim \nn(A^*)$, which in turn is equivalent to statement (ii).
\end{proof}

\begin{proposition}\label{pro550} Let $\hh$ be a Hilbert space and consider $A\in \ll (\hh)$ such that
$A$ is Moore-Penrose invertible and the codimension of $\rr(A)$ is finite. Suppose that there exists a sequence of 
operators $\suc{A_n}\subset \ll (\hh )$ such that
for each $n\in\mathbb{N}$, $A_n$ is Moore-Penrose invertible and  $\suc{A_n}$ converges to $A$. Then
the following statements are equivalent.
\begin{enumerate}[{\rm (i)}]
\item  The sequence $\suc{A^\dag_n}$ converges to $A^\dag$.
\item For all sufficiently large $n\in\ene$, $\hbox{\rm codim }\rr(A_n)=\hbox{\rm codim }\rr(A)$.
\end{enumerate}
\end{proposition}
\begin{proof} Apply a similar argument to the one in the proof of Proposition \ref{pro540}, using in particular  \cite[Corollary 10]{rs} instead of \cite[Theorem 3]{rs}.
\end{proof}

\begin{corollary}\label{cor560}Let $\hh$ be a Hilbert space and consider $A\in \ll (\hh)$ such that
$A$ is group invertible and either the codimension of $\rr(A)$ is finite or $\dim \nn(A)$ is finite. Suppose that there exists a sequence of 
operators $\suc{A_n}\subset \ll (\hh )$ such that
for each $n\in\mathbb{N}$, $A_n$ is group invertible and  $\suc{A_n}$ converges to $A$. Then,
the following statements are equivalent.
\begin{enumerate}[{\rm (i)}]
 \item The sequence $\suc{A^\#_n}$ converges to $A^\#$.\par
\item  The sequence $\suc{A^\dag_n}$ converges to $A^\dag$.
\end{enumerate}
\end{corollary}
\begin{proof} Recall that given and operator $S\in\ll (\hh)$ such that $S$ is group invertible, then $S$ is Moore-Penrose invertible (\hspace{-1pt}\cite[Theorem 6]{hm1}).
To conclude the proof apply, when $\dim\nn(A)$ is finite,   \cite[Theorem 3]{rs} and \cite[Corollary 10]{rs}, and when codimension of $\rr(A)$ is finite, Proposition \ref{pro540}
and Proposition \ref{pro550}.
\end{proof}

\indent Now a characterization of the continuity of the (dual) core inverse for a particular case of  Hilbert spaces operators will be presented.

\begin{theorem} \label{thm570}Let $\hh$ be a Hilbert  space and consider $A\in \ll (\hh)$ such that
$A$ is (dual) core invertible and either the codimension of $\rr(A)$ is finite or $\dim \nn(A)$ is finite. Suppose that there exists a sequence of 
operators $\suc{A_n}\subset \ll (\hh )$ such that
for each $n\in\ene$, $A_n$ is (dual) core invertible and  $\suc{A_n}$ converges to $A$. The following statements are equivalent.

\begin{enumerate}[{\rm (i)}]
\item  The sequence $\suc{\core{A}_n}$ converges to $\core{A}$.
\item  The sequence $\suc{\dcore{A_n}}$ converges to $\dcore{A}$.
\item The sequence  $\suc{A^\dag_n}$ converges to $A^\dag$.

\hspace*{\dimexpr\linewidth-\textwidth\relax}{\noindent When $\dim \nn (A)$ is finite, statements {\rm (i)}-{\rm (iii)} are equivalent to the following statement. }
\item  For all sufficiently large $n\in\ene$, $\dim \nn(A_n)=\dim \nn(A)$.

\hspace*{\dimexpr\linewidth-\textwidth\relax}{\noindent When $\hbox{\rm codim } (A)$ is finite, statements {\rm (i)}-{\rm (iii)} are equivalent to the following statement. }

\item For all sufficiently large $n\in\ene$, $\hbox{\rm codim } \rr(A_n)=\hbox{\rm codim } \rr(A)$.
\end{enumerate}
\end{theorem}
\begin{proof} Apply Theorem \ref{thm320}, Corollary \ref{cor560}, \cite[Theorem 3]{rs} and  Proposition \ref{pro540}.
For the case $A=0$, apply Remark \ref{rem3950} and Proposition \ref{prop3960}.
\end{proof}

\indent Now the finite dimensional case will be derived from Theorem \ref{thm570}. It is worth noticing that the following corollary
also provides a different proof of a well known result concerning the continuity of the Moore-Penrose inverse in the matricial setting, see \cite[Theorem 5.2]{S}.

\begin{corollary}\label{cor580}
Let $A \in \ce_m$ be a (dual) core invertible matrix. Suppose that exists a sequence $\suc{A_n} \subset \ce_m$
of (dual) core invertible matrices such that $\suc{A_n}$ converges to $A$. The following statements are equivalent.
\begin{enumerate}[{\rm (i)}]
\item The sequence $\suc{\core{A}_n}$ converges to $\core{A}$.
\item The sequence $\suc{\dcore{A_n}}$ converges to $\dcore{A}$.
\item The sequence $\suc{A_n^\dag}$ converges to $A^\dag$.
\item There exists $n_0 \in \ene$ such that $\rk(A_n) = \rk(A)$, for $n \geq n_0$.
\end{enumerate}
\end{corollary}
\begin{proof} Apply Theorem \ref{thm570}.
\end{proof}

\section{Differentiability of the (dual) core inverse}

\noindent To prove the main results of this section, some preparation is needed.

Let $U\subseteq \ere$ be an open set and consider ${\bold a}\colon U\to\aa$ a function such that ${\bold a}(U)\subseteq\core{\aa}$. Since according to Theorem \ref{theo310},  $\core{\aa}=\dcore{\aa}=\aa^\#\subset\aa^\dag$,
it is possible to consider the functions  
$$
\core{{\bold a}},\, \dcore{{\bold a}},\, {\bold a}^\#,\, {\bold a}^\dag\colon U\to\aa,
$$ 
\noindent which are defined as follows. Given $u\in U$,
\begin{align*}
& &&\core{{\bold a}}(u)=\core{({\bold a}(u))},&&\dcore{{\bold a}}(u)=\dcore{({\bold a}(u))},&\\
& &&{\bold a}^\#(u)=({\bold a}(u))^\#,&&{\bold a}^\dag (u)= ({\bold a}(u))^\dag.&\\
\end{align*}

Since in this section functions instead of sequence will be considered and the notion of continuity will be central in the results concerning differentiability,
Theorem \ref{thm320} will be reformulated for functions.

\begin{theorem}\label{thm530} Let $\aa$ be a unital $C^*$-algebra and consider $U\subseteq \ere$ an open set and a function ${\bold a}\colon U\to\aa$ such that 
${\bold a}(U)\subseteq\core{\aa}$ and ${\bold a}$ is continuous at $t_0\in U$.
The following statements are equivalent.
\begin{enumerate}[{\rm (i)}]
\item  The element ${\bold a}(t_0)\in\core{\aa}$ and the function $\core{{\bold a}}$ is continuous at $t_0$.
\item  The element ${\bold a}(t_0)\in\dcore{\aa}$ and the function $\dcore{{\bold a}}$ is continuous at $t_0$.
\item  The element ${\bold a}(t_0)\in\aa^\#$ and the function ${\bold a}^\#$ is continuous at $t_0$.
\item   The element ${\bold a}(t_0)\in\core{\aa}$ and there exists an open set $V\subseteq U$ such that $t_0\in V$ and the function $\core{{\bold a}}$ is  bounded on $V$.
\item    The element ${\bold a}(t_0)\in\dcore{\aa}$ and there exists an open set $W\subseteq U$ such that $t_0\in W$ and the function $\dcore{{\bold a}}$ is  bounded on $W$.
\item  The element ${\bold a}(t_0)\in\aa^\dag$,  the function ${\bold a}^\dag$ is continuous at $t_0$, and there exists and open set $I\subseteq U$ such that $t_0\in I$
and  the function $\core{{\bold a}}{\bold a}$ is bounded on $I$.
\item The element ${\bold a}(t_0)\in\aa^\dag$,  the function ${\bold a}^\dag$ is continuous at $t_0$ , and there exists and open set $J\subseteq U$ such that $t_0\in J$
and the function ${\bold a}\core{{\bold a}}$ is bounded on $J$.
\item  The element ${\bold a}(t_0)\in\aa^\dag$,  the function ${\bold a}^\dag$ is continuous at $t_0$, and there exist an open set $Z$ such that $t_0\in Z$and $\psi\in[0, \frac{\pi}{2})$ such that 
when ${\bold a}(t)\neq 0$ ($t\in Z$), $\psi_t=\psi_{{\bold a}(t)}\le \psi$.
\end{enumerate}
\end{theorem}
\begin{proof} Apply Theorem \ref{thm320}.
\end{proof}

\begin{remark}\label{rem591}\rm  Note that under the same hypotheses of Theorem \ref{thm530}, when ${\bold a}(t_0)=0$, the continuity of the  function $\core{{\bold a}}$
(respectively $\dcore{{\bold a}}$, ${\bold a}^\#$,  ${\bold a}^\dag$) at $t_0$ is equivalent to the following condition: there exists an open set $K\subseteq U$, $t_0\in K$
and ${\bold a}(t)=0$, for all $t\in K$ (Remark \ref{rem3950}, Proposition \ref{prop3960}).
\end{remark}
\indent To study the differentiability of the (dual) core inverse, the differentiability of the Moore-Penrose inverse need to be  considered first.

\begin{remark}\label{rem540} \rm Let $\aa$ be a unital $C^*$-algebra and consider an open set $U$ and  ${\bold a} \colon U\to \aa$ a function such that
 ${\bold a}(U)\subset \aa^\dag$ and there is $t_0$ such that ${\bold a}$ is differentiable at $t_0$. Thus, a necessary and sufficient condition for ${\bold a}^\dag$ to be 
differentiable at $t_0$ is that ${\bold a}^\dag$ is continuous at $t_0$. In fact, if ${\bold a}(t_0)\neq 0$, there is an open set $V\subseteq U$ such that $t_0\in V$ and ${\bold a}(t)\neq 0$ for $t\in V$, and then according to 
\cite[Theorem 2.1]{kol}, this equivalence holds. On the other hand, if   ${\bold a}(t_0)=0$, according to Remark \ref{rem3950} (vi)-(vii), the function ${\bold a}^\dag$ is continuous at $t_0$ if and only if there exists an open set $W$ such that $t_0\in  W$ and ${\bold a}(t)=0$ for $t\in W$, which implies that ${\bold a}^\dag$ is differentiable at $t_0$. As a result, in \cite[Theorem 2.1]{kol} it is not necessary to assume that ${\bold a}(t)\neq 0$ for $t$ in a neighbourhood of $t_0$.
\end{remark}
\indent In the following theorem the differentiability of the (dual) core inverse will be studied. Note that the following notation will be used.
Given a unital $C^*$-algebra $\aa$, if $U\subseteq \ere$ is an open set and ${\bold b}\colon U\to\aa$ is a function, then ${\bold b}^*\colon U\to\aa$
will denote the function ${\bold b}^*(t)=({\bold b}(t))^*$ ($t\in U$). In addition, if ${\bold b}\colon U\to\aa$ is differentiable at $t_0\in U$, 
then ${\bold b}'(t_0)$ will stand for the derivative of ${\bold b}$ at $t_0$.

\begin{theorem}\label{thm510} Let $\aa$ be a unital $C^*$-algebra and consider $U\subseteq \ere$ an open set and ${\bold a}\colon U\to\aa$ a function such that  is
differentiable at $t_0\in U$ and ${\bold a}(U)\subset \core{\aa}=\dcore{\aa}=\aa^\#$. The following statements are equivalent.
\begin{enumerate}[{\rm (i)}]
\item The function $\core{{\bold a}}$ is continuous at $t_0$.
\item The function $\core{{\bold a}}$ is differentiable at $t_0$.
\item The function $\dcore{{\bold a}}$ is differentiable at $t_0$.
\item The function ${\bold a}^\#$ is differentiable at $t_0$.

\hspace*{\dimexpr\linewidth-\textwidth\relax}{\noindent Furthermore, the following formulas hold.}
\item
\begin{align*}
(\core{{\bold a}})'(t_0)&=\core{{\bold a}}(t_0){\bold a}(t_0)({\bold a}^\dag)'(t_0)(\uno-{\bold a}(t_0)\core{{\bold a}}(t_0))- \core{{\bold a}}(t_0){\bold a}'(t_0)\core{{\bold a}}(t_0)\\
&+(\uno-\core{{\bold a}}(t_0){\bold a}(t_0)){\bold a}'(t_0){\bold a}^\dag(t_0)\core{{\bold a}}(t_0).\\
\end{align*}
\item 
\begin{align*}
(\dcore{{\bold a}})'(t_0)&=(\uno-\dcore{{\bold a}}(t_0){\bold a}(t_0))({\bold a}^\dag)'(t_0){\bold a}(t_0)\dcore{{\bold a}}(t_0)   - \dcore{{\bold a}}(t_0){\bold a}'(t_0)\dcore{{\bold a}}(t_0)\\
&+\dcore{{\bold a}}(t_0){\bold a}^\dag(t_0){\bold a}'(t_0)(\uno-{\bold a}(t_0)\dcore{{\bold a}}(t_0)).\\
\end{align*}
\item 
\begin{align*}
({\bold a}^\#)'(t_0)&=2\core{{\bold a}}(t_0)(\core{{\bold a}})'(t_0){\bold a}(t_0)+ (\core{{\bold a}}(t_0))^2{\bold a}'(t_0)\\
&={\bold a}'(t_0)(\dcore{{\bold a}}(t_0))^2 +2{\bold a}(t_0)\dcore{{\bold a}}(t_0)(\dcore{{\bold a}})'(t_0)\\
&= (\core{{\bold a}})'(t_0){\bold a}(t_0)\dcore{{\bold a}}(t_0)+\core{{\bold a}}(t_0){\bold a}'(t_0)\dcore{{\bold a}}(t_0)+\core{{\bold a}}(t_0){\bold a}(t_0)(\dcore{{\bold a}})'(t_0).
\end{align*}
\end{enumerate}
\end{theorem}
\begin{proof} According to Lemma \ref{lema4}, 
\begin{align*}
\core{{\bold a}}(t)-\core{{\bold a}}(t_0)&= \core{{\bold a}}(t){\bold a}(t)({\bold a}^\dag(t)-{\bold a}^\dag (t_0))(\uno-{\bold a}(t_0)\core{{\bold a}}(t_0))\\ 
&+\core{{\bold a}}(t)({\bold a}(t_0)-{\bold a}(t))\core{{\bold a}}(t_0)
+ (\uno-\core{{\bold a}}(t){\bold a}(t))({\bold a}(t)-{\bold a}(t_0)){\bold a}^\dag(t_0)\core{{\bold a}}(t_0).\\
\end{align*}

Now suppose that statement (i) holds. According to Theorem \ref{thm530}, the function  ${\bold a}^\dag$ is continuous at $t_0$, and according to 
\cite[Theorem 2.1]{kol} and Remark \ref{rem540}, the function ${\bold a}^\dag$ is differentiable at $t_0$. Thus,
$$\frac{\core{{\bold a}}(t){\bold a}(t)({\bold a}^\dag (t)-{\bold a}^\dag (t_0))(\uno-{\bold a}(t_0)\core{{\bold a}}(t_0))}{t-t_0}$$ 
\noindent converges to $\core{{\bold a}}(t_0){\bold a}(t_0)({\bold a}^\dag)'(t_0)(\uno-{\bold a}(t_0)\core{{\bold a}}(t_0))$. In addition,
$$
\frac{\core{{\bold a}}(t)({\bold a}(t_0)-{\bold a}(t))\core{{\bold a}}(t_0)}{t-t_0}
$$
\noindent converges to  $- \core{{\bold a}}(t_0){\bold a}'(t_0)\core{{\bold a}}(t_0)$, and 
$$
\frac{(\uno-\core{{\bold a}}(t){\bold a}(t))({\bold a}(t)-{\bold a}(t_0)){\bold a}^\dag(t_0)\core{{\bold a}}(t_0)}{t-t_0}
$$
\noindent converges to $(\uno-\core{{\bold a}}(t_0){\bold a}(t_0)){\bold a}'(t_0){\bold a}^\dag(t_0)\core{{\bold a}}(t_0)$. Consequently statements (ii) and (v) hold. 

It is evident that statement (ii) implies statement (i). 

Now observe that the function ${\bold a}^*\colon U\to\aa$ is differentiable at $t_0$ and ${\bold a}^*(U)\subset \core{\aa}$ (Theorem \ref{theo310}). 

Suppose that statement (i) holds. According to the identity $\core{({\bold a}^*)}(t)=(\dcore{{\bold a}})^*(t)$ and Theorem \ref{thm530},
the function $\core{({\bold a}^*)}\colon U\to \aa$ is continuous at $t_0$.
Thus, according to what has been proved, the function  $\core{({\bold a}^*)}\colon U\to \aa$ is 
differentiable at $t_0$. Therefore, the function $\dcore{{\bold a}}\colon U\to \dcore{\aa}$ is differentiable at $t_0$. Consequently, statement (iii) holds.
Furthermore, since  $(\dcore{{\bold a}})'(t_0)=((\core{{(\bold a})^*})')^*(t_0)$, to prove statement (vi), apply statement (v).

On the other hand, if statement (iii) holds, then the function $\dcore{{\bold a}}$ is continuous at $t_0$. According to Theorem \ref{thm530}, statement (i) holds.

Suppose that statement (i) holds. According to \cite[Theorem 2.19]{rdd} and Lemma \ref{lema2} (iv), the following identities hold.
$$
{\bold a}^\#=(\core{{\bold a}})^2{\bold a}={\bold a}(\dcore{{\bold a}})^2=\core{{\bold a}}{\bold a}\dcore{{\bold a}}.
$$
Therefore, according to what has been proved, the function ${\bold a}^\#$ is differentiable at $t_0$. Furthermore, from these idenetities statement (vii) can be derived.

On the other hand, according to Theorem \ref{thm530}, statement (iv) implies statement (i).
\end{proof}

\begin{remark}\label{rem570}\rm Under the same hypothesis of Theorem \ref{thm510}, the following facts should be noted.\par
\noindent (i). When $a(t_0)=0$, according to Remark \ref{rem591}, 
$$
(\core{{\bold a}})'(t_0)=(\dcore{{\bold a}})'(t_0)=({\bold a}^\#)'(t_0)=({\bold a}^\dag)' (t_0)=0.
$$
\noindent (ii). Recall that in \cite[Theorem 2.1]{kol},
a formula concerning the derivative of the function ${\bold a}^\dag$ at $t_0$ was given.\par
\noindent (iii). Note that according to Theorem \ref{thm530}, a necessary and sufficient condition for the function $\dcore{{\bold a}}$
(respectively ${\bold a}^\#$) to be differentiable at $t_0$ is that  $\dcore{{\bold a}}$ (respectively ${\bold a}^\#$) is continous at $t_0$.
In fact, the continuity of one of the functions $\core{{\bold a}}$, $\dcore{{\bold a}}$ and ${\bold a}^\#$ at a point $t_0$ is equivalent to the
continuity and the differentiability of the three functions under consideration at $t_0$ (Theorem \ref{thm530} and Theorem \ref{thm510}).\par
\noindent (iv). According to \cite[Theorem 2.19]{rdd}, 
\begin{align*}
& &&{\bold a}^\dag=\dcore{{\bold a}}{\bold a}\core{{\bold a}},& &\core{{\bold a}}= {\bold a}^\#{\bold a}{\bold a}^\dag,&
&\dcore{{\bold a}}={\bold a}^\dag {\bold a}{\bold a}^\#.&\\
\end{align*}
\noindent Thus, the derivative of ${\bold a}^\dag$, $\core{{\bold a}}$ and $\dcore{{\bold a}}$
at $t_0$ can also be computed as follows:
\begin{align*}
({\bold a}^\dag)'(t_0)&= (\dcore{{\bold a}})'(t_0){\bold a}(t_0)\core{{\bold a}}(t_0)+ \dcore{{\bold a}}(t_0){\bold a}'(t_0)\core{{\bold a}}(t_0) +\dcore{{\bold a}}(t_0){\bold a}(t_0)(\core{{\bold a}})'(t_0).\\
(\core{{\bold a}})'(t_0)&=({\bold a}^\#)'(t_0){\bold a}(t_0){\bold a}^\dag(t_0)+ {\bold a}^\#(t_0){\bold a}'(t_0){\bold a}^\dag(t_0)+{\bold a}^\#(t_0){\bold a}(t_0)({\bold a}^\dag)'(t_0).\\
(\dcore{{\bold a}})'(t_0)&=({\bold a}^\dag)'(t_0){\bold a}(t_0){\bold a}^\#(t_0)+{\bold a}^\dag(t_0){\bold a}'(t_0){\bold a}^\#(t_0)+ {\bold a}^\dag(t_0){\bold a}(t_0) ({\bold a}^\#)'(t_0).\\
\end{align*}
\end{remark}

\vskip.3truecm
\noindent Julio Ben\'{\i}tez\par
\noindent E-mail address: jbenitez@mat.upv.es
\vskip.3truecm
\noindent Enrico Boasso\par
\noindent E-mail address: enrico\_odisseo@yahoo.it \par
\vskip.3truecm
\noindent Sanzhang Xu\par
\noindent xusanzhang5222@126.com 

\begin{thebibliography}{99} 

\bibitem{BT} O. M. Baksalary, G. Trenkler, Core inverse of matrices, Linear Multilinear Algebra 58 (2010) 681--697.

\bibitem{cbl} J. Ben\'{\i}tez, X. Liu, On the continuity of the group inverse, Oper. Matrices 6 (2012) 859-868.

\bibitem{bc} J. Ben\'{\i}tez, D. Cvetkovi\'c-Ili\'c, On the elements $aa^\dag$ and $a^\dag a$ in a ring, Appl. Math. Comput. 222 (2013) 478--489.

\bibitem{cbl2} J. Ben\'{\i}tez, D. Cvetkovi\'c-Ili\'c, X. Liu,  On the continuity of the group inverse in $C^*$-algebras 
	Banach J. Math. Anal. 8 (2014) 204--213.

\bibitem{b} E. Boasso, Drazin spectra of Banach space operators and Banach algebra elements, 
	J. Math. Anal. Appl. 359 (2009) 48--55.

\bibitem{b2} E. Boasso, G. Kant\'un-Montiel, The $(b, c)$-inverse in rings and in the Banach context, 
Mediterr. J. Math. 14 (2017), doi:10.1007/s00009-017-0910-1.


\bibitem{dr} D. Djordjevi\'c, V. Rako\v{c}evi\'c, Lectures on Generalized Inverses,
	Faculty of Sciences and Mathematics, University of Ni\v{s}, Ni\v{s}, Serbia, 2008.

\bibitem{Do} R. G. Douglas, On majorization, factorization and range inclusion of operators
on Hilbert spaces, Proc. Amer. Math. Soc. 17 (1966) 413-415.

\bibitem{D} M. P. Drazin,  A class of outer generalized inverses, Linear Algebra Appl. 436 (2012) 1909-1923.

\bibitem{dx} F. Du, Y. Xue, Perturbation analysis of $A_{T,S}^{(2)}$ on Banach spaces,
	Electron. J. Linear Algebra 23 (2012) 586--598.

\bibitem{hm1} R.E.  Harte, M.  Mbekhta, On generalized  inverses  in $C^*$-algebras,
	Studia  Math. 103 (1992) 71-77.

\bibitem{hm} R.E.  Harte, M.  Mbekhta, Generalized  inverses  in $C^*$-algebras  II, 
	Studia  Math. 106 (1993) 129--138.

\bibitem{kato} T. Kato, Perturbation Theory for Linear Operators,
	Springer Verlarg, Berlin, Heidelberg, New York, 1980.

\bibitem{kol} J. J.Koliha, Continuity and diffrentiability of the Moore-Penrose inverse in $C^*$-algebras,
Math. Scand. 88 (2001) 154-160. 

\bibitem{kr1} J.J. Koliha, V. Rako\v{c}evi\'c, Continuity of the Drazin inverse II, Studia Math. 131 (1998) 167-177.

\bibitem{kr} J.J. Koliha, V. Rako\v{c}evi\'c, On the norm of idempotents in $C^*$-algebras,
	Rocky Mountain J. Math. 34 (2004) 685--697.

\bibitem{mb} M. Mbekhta, Conorme et inverse g\'en\'eralis\'e dans les $C^*$-alg\`ebres,
	Canad. Math. Bull. 35 (1992) 515--522.

\bibitem{P} R. Penrose, A generalized inverse for matrices, Math. Proc. Cambridge Philos. Soc. 3 (1955) 406-413.

\bibitem{r} D. S. Raki\'c, A note on Rao and Mitra's constrained inverse and Drazin's $(b, c)$ inverse,
Linear Algebra Appl. 523 (2017) 102-108.

\bibitem{rdd} D. S. Raki\'c, N. \v{C}. Din\v{c}i\'c, D. S. Djordjevi\'c,
	Group, Moore-Penrose, core and dual core inverse in rings with involution,
	Linear Algebra Appl. 463 (2014) 115--133.

\bibitem{rdd2} D. S. Raki\'c, N. \v{C}. Din\v{c}i\'c, D. S. Djordjevi\'c, Core inverse and core partial order of Hilbert space operators,
Appl. Math. Comp. 244 (2014) 283-302.

\bibitem{V2} V. Rako\v{c}evi\'c,  On the continuity of the Moore-Penrose inverse in Banach algebras,
Facta Univ. Ser. Math. Inform. 6 (1991)  133-138.

\bibitem{V} V. Rako\v{c}evi\'c, Continuity of the Drazin inverse, J. Operator Theory 41 (1999)  55-68.

\bibitem{rs} S. Roch, B. Silbermann, Continuity of generalized inverses in Banach algebras,
	Studia Math. 136 (1999) 197--266.

\bibitem{S} G. W. Stewart, On the continuity of the generalized inverse, SIAM J. Appl. Math. 17 (1969) 33-45.

\bibitem{14}  S. Z. Xu,  J.L. Chen,  X. X. Zhang, New characterizations for core and dual core inverses in rings with involution, 
	Front. Math. China  12 (2017) 231--246. 

\end{thebibliography}
\end{document}